 \documentclass[11pt]{amsart}
\usepackage{amssymb,amsmath}
\usepackage{breqn,cleveref}
 \usepackage{comment}

\font\bbbld=msbm10 scaled\magstephalf

\newcommand{\bi}{\bar{i}}
\newcommand{\bj}{\bar{j}}
\newcommand{\bk}{\bar{k}}
\newcommand{\bl}{\bar{l}}
\newcommand{\bm}{\bar{m}}
\newcommand{\bn}{\bar{n}}

\newcommand{\bq}{\bar{q}}

\newcommand{\bz}{\bar{z}}

\newcommand{\bJ}{\bar{J}}
\newcommand{\bL}{\bar{L}}

\newcommand{\bI}{\bar{I}}
\newcommand{\bK}{\bar{K}}

\newcommand{\balpha}{\bar{\alpha}}

\newcommand{\bzeta}{\bar{\zeta}}

\newcommand{\bpartial}{\bar{\partial}}

\newcommand{\fg}{\mathfrak{g}}
\newcommand{\fI}{\mathfrak{I}}

\newcommand{\fRe}{\mathfrak{Re}}

\newcommand{\bfC}{\hbox{\bbbld C}}
\newcommand{\bfG}{\hbox{\bbbld G}}

\newcommand{\bfN}{\hbox{\bbbld N}}

\newcommand{\bfR}{\hbox{\bbbld R}}

\newcommand{\cC}{\mathcal{C}}

\newcommand{\cL}{\mathcal{L}}

\newcommand{\ol}{\overline}

\newtheorem{theorem}{Theorem}[section]
\newtheorem{lemma}[theorem]{Lemma}
\newtheorem{proposition}[theorem]{Proposition}
\newtheorem{corollary}[theorem]{Corollary}

 \theoremstyle{definition}
\newtheorem{definition}[theorem]{Definition}

\theoremstyle{remark}
\newtheorem{remark}{Remark}

\numberwithin{equation}{section}

\begin{document}

\title[fully nonlinear Parabolic equations]
{Fully nonlinear parabolic equations of real forms on Hermitian manifolds}
\author{Mathew George}
\address{Department of Mathematics, Purdue University,
         West Lafayette, IN 47907 }
\email{georg233@purdue.edu}
\author{Bo Guan}
\address{Department of Mathematics, Ohio State University,
         Columbus, OH 43210, USA}

\email{guan@math.ohio-state.edu}

\date{}

\begin{abstract} 
Over many decades fully nonlinear PDEs, and the complex Monge-Amp\`ere equation in particular 
played a central role in the study of complex manifolds. 
Most previous works focused on problems that can be expressed through equations involving real $(1,1)$ forms.
As many important questions, especially those linked to higher cohomology classes in complex geometry involve real $(p, p)$ forms for $p > 1$, there is a strong need to develop PDE techniques to study them. 
In this paper we consider a fully nonlinear equation for $(p, p)$ forms on compact Hermitian manifolds.
We establish the existence of classical solutions for a large class of these equations
by a parabolic approach, proving the long-time existence and convergence of solutions to the elliptic case.


{\em Mathematical Subject Classification (MSC2020):}
35K10, 35K55, 53C55, 58J35.
\end{abstract}

\maketitle

\section{Introduction}
\label{gg-I}
\setcounter{equation}{0}
\medskip

In the landmark work of Chern~\cite{Chern46} certain differential forms constructed from Hermitian metrics seemed
to effectively describe geometric and topological information associated with complex manifolds, which are now called Chern classes. Studies of these characteristic classes have been undertaken from different points of view in geometry and topology. It would be desirable to also have strong analytical tools to study them. Motivated by this, we study a class of nonlinear PDEs involving real $(p,p)$ forms on Hermitian manifolds.

The first Chern class for a compact K\"ahler manifold consists of closed real $(1,1)$ forms and is 
represented by the Ricci form of the metric. The Calabi-Yau theorem  
~\cite{Yau78} asserts that any $(1,1)$ form in the first Chern class of a compact K\"ahler manifold $(M^n, \omega)$ corresponds to the Ricci form of some K\"ahler metric on $M$ 
in the cohomology class of $\omega$. Initially conjectured by Calabi, this was proved by Yau~\cite{Yau78} by solving the complex Monge-Amp\`ere equation on compact K\"ahler manifolds.
 Since the works of Yau~\cite{Yau78} and Aubin~\cite{Aubin78},
 the complex Monge-Amp\`ere equation has been a powerful tool in the study of complex geometry, and 
 more recently, there have been growing interests in fully nonlinear PDEs on complex manifolds, such as the J-equation and deformed Hermitian Yang-Mills (dHYM) equation.
 However, most investigations have been focused on equations or systems involving real 
 $(1,1)$  forms.
 Because of their important role,  
especially in the study of higher cohomology classes, it is natural to extend the nonlinear PDE theory to equations of higher order forms, more precisely real $(p, p)$ forms for $p > 1$. This is our primary goal in the current paper,
and we shall take a parabolic approach.

Nonlinear parabolic methods have become increasingly prevalent in complex geometry.
In 1985, Cao~\cite{Cao1985} first studied K\"ahler-Ricci flow and using which gave an alternative proof of Calabi-Yau theorem. Donaldson~\cite{Donaldson85} used nonlinear parabolic equations to establish the Kobayashi-Hitchin correspondence between stable vector bundles on projective algebraic surfaces and Hermitian-Einstein connections.
Later on Uhlenbeck and Yau~\cite{UY86} extended the results to compact K\"ahler manifolds.
In 1999, Donaldson~\cite{Donaldson99a} proposed the study of J-flow on K\"ahler manifolds, which seemed the first 
fully nonlinear PDE other than the complex Monge-Amp\`ere equation that appeared in complex geometry. Song and Weinkove~\cite{SW08} solved J-flow by finding a key necessary and sufficient cone condition for the solvability
(see also~\cite{Weinkove04, Weinkove06}); their results were generalized by
Fang-Lai-Ma~\cite{FLM11} 
to Hessian quotient equations. 
On Hermitian manifolds, Gill~\cite{Gill11} first considered the parallel of K\"ahler-Ricci flow, followed by  
Sun~\cite{Sun15} who treated complex Hessian quotient parabolic equations,
extending results in \cite{FLM11} and ~\cite{Gill11},
while Streets and Tian~\cite{ST10, ST11, ST13} studied curvature flows determined by the second Chern-Ricci form. More recently, Phong-Picard-Zhang~\cite{PPZ2, PPZ4, PPZ19} introduced new geometric flows,
while Phong and T\^{o}~\cite{PT} treated fully nonlinear parabolic equations of the form
\begin{equation}
\label{gg-PI10}
\frac{\partial \phi}{\partial t} = f (\lambda (\chi + \sqrt{-1} \partial \bpartial \phi)) - \psi
\end{equation}
on compact Hermitian manifolds 
 in a general framework established by Caffarelli-Nirenberg-Spruck~\cite{CNS3} for elliptic equations. 

In this paper we consider a new class of fully nonlinear parabolic equations which have significant differences from equation~\eqref{gg-PI10}. 
 The geometric considerations while setting up these equations are twofold. Firstly, many natural invariants of a complex manifold are expressed in terms of real $(p,p)$ forms for $p>1$, and secondly in most cases these forms are defined by certain symmetric relations. This leads us to consider the following class of equations

Let $(M^n, \omega)$ be a compact Hermitian manifold of dimension $n \geq 2$,
and $\Omega$ a real $(p,p)$ form on $M$ where $p$ is an integer, $2 \leq p \leq n-1$. 
As we shall see in Section~\ref{PR}, there is a natural way to define the eigenvalues of 
$(p, p)$ forms with respect to the metric $\omega$. 
For a function $\phi \in C^{\infty} (M)$ 
let 
\[ \Lambda (\Omega_\phi) \equiv 
   \Lambda (\Omega + \sqrt{-1} \partial \bpartial \phi \wedge \omega^{p-1}) 
= (\Lambda_1, \ldots, \Lambda_N), \]
where
\[ N = \frac{n!}{p! (n-p)!}, \]
denote the {eigenvalues} of $\Omega_{\phi} = \Omega + \sqrt{-1}\partial \bpartial \phi \wedge \omega^{p-1}$ 
with respect to $\omega$. Let 
$f$ be a symmetric function of $N$ variables.
We shall consider the equation 
\begin{equation}
\label{I1}
\begin{aligned}
  \frac{\partial \phi}{\partial t} 
  =  f (\Lambda(\Omega + \sqrt{-1}\partial \bpartial \phi \wedge \omega^{p-1})) - \psi 
    \end{aligned}
\end{equation}
where, in geometric applications, $\Omega$ often depends on $\phi$ and its gradient, with the initial condition
\begin{equation}
\label{gg-I20}
\begin{aligned}
    &\phi(z, 0)=\phi_0\in C^{\infty}(M).
    \end{aligned}
\end{equation}

Note that if $\omega$ is K\"ahler, that is $d \omega = 0$, and $\Omega$ is closed, then 
$\Omega_{\phi}$ is also closed and belongs to the Bott-Chern cohomology class of $\Omega$.
So equation~~\eqref{I1} preserves Bott-Chern cohomology classes.

When $p = 1$, equation~\eqref{I1} reduces to \eqref{gg-PI10}.
In spite of the common background they share in terms of the structure conditions described below, 
for general $p > 1$ equation~\eqref{I1} is significantly different from equation~\eqref{gg-PI10}, which makes it 
much more difficult to study, and seems to have rarely been treated before. 
To the best of our knowledge, the first fully nonlinear equation involving real $(p, p)$ forms for $p > 1$
was the form type Calabi-Yau equation for positive $(n-1, n-1)$ forms on balanced manifolds 
introduced by Fu-Wang-Wu~\cite{FWW11, FWW15}; see also 
Tosatti-Weinkove~\cite{TW17, TW2019} and their work with Sz\'ekelyhidi~\cite{STW17} which resolved the Gauduchon conjecture, extending the Calabi-Yau Theorem to non-K\"ahler metrics. 
 The approach in these papers was based on a natural duality between $(n-1, n-1)$ forms and $(1,1)$ forms,
 and the fact that a positive
 $(n-1, n-1)$ form $\Omega$ admits the $(n-1)$-th root $\eta$, thereby $\Omega = \eta^{n-1}$. Through the Hodge star operator, their equations reduce to the elliptic version of ~\eqref{gg-PI10} where $\chi$ also depends linearly on the gradient of $\phi$, for $(1,1)$ forms.

We shall follow \cite{CNS3} to assume that $f$ is a symmetric function defined in a symmetric open convex cone $\Gamma \subset \mathbb{R}^N$ 
(of higher dimension $N > n$ unless $p = 1$, or $n-1$) with vertex at the origin 
containing the positive cone
\begin{equation}
\label{P1}
    \Gamma_N= \{\Lambda\in \mathbb{R}^N: \Lambda_i > 0\}\subset\Gamma,
\end{equation}
and satisfies the 
\begin{equation}
\label{P2}
    f_i \equiv \frac{\partial f}{\partial \Lambda_i} \geq 0  
       \text{ in $\Gamma$, $1\leq i\leq N$},
\end{equation}
\begin{equation}
\label{P3}
    \text{$f$ is a concave function in $\Gamma$}
\end{equation}
and
\begin{equation}
\label{P4}
    \sup\limits_{\partial\Gamma}f < \inf\limits_M \psi. 
\end{equation}
We call a function $\phi$ \textit{admissible} if 
$\Lambda (\Omega_{\phi}) \in \Gamma$ in $M \times \{t > 0\}$
so that equation~\eqref{I1} is degenerate parabolic for admissible solutions. %

\begin{proposition}
\label{gg-prop-I10}
Under condition~\eqref{P2}, equation~\eqref{I1} is degenerate parabolic with respect to admissible solutions.
\end{proposition} 

For general $p > 1$ this is less obvious compared to the case $p = 1$; see Section~\ref{PR} for more details.

Our main focus in this paper will be on equation~\eqref{I1} 
with 
\begin{equation}
\label{P-gg10}
 \Omega = X [\phi] \equiv X (z, \phi, \partial \phi, \bar \partial \phi) 
 \end{equation}
being allowed to depend on $\phi$ and its gradient.
To give some geometric motivations which lead to this more general equation, 
we call a $(p, p)$ form $\Omega$
 {\sl pluriclosed} if $\partial \bar \partial \Omega = 0$, and a $(1,1)$ form $\eta$ is called $p$-{\sl pluriclosed} if $\eta^p$ is pluriclosed.
For $p = 1$, pluriclosed metrics are also called {\em Strong KT metrics} and were extensively studied; see e.g. \cite{Bismut89, FT09, ST10, ST13, ZZ23} and references therein,
while $p$-pluriclosed $(1,1)$ forms are known as Gauduchon~\cite{Gauduchon84} and 
Astheno-K\"ahler~\cite{JY93} for $p = n-1$  and $n-2$, respectively.
Recall that the Aeppli cohomology group is defined as
\[ H^{p,p}_{A} (M, \bfR) = \frac{\{\mbox{pluriclosed real $(p,p)$ forms}\}}
{\{\partial \alpha + \bpartial \balpha:  \alpha\in \Lambda^{p-1, p} (M, \bfR)\}}. \]


Let $\Omega$ be a smooth real $(p,p)$ form and $\gamma \in \bfR$. 
Adopting an idea in 
\cite{STW17} we define for $\phi \in C^2 (M)$
\begin{equation}
\label{gg-I10ab}
X [\phi] = \Omega
   + \frac{\gamma \sqrt{-1}}{2} \partial \phi \wedge \bpartial \omega^{p-1}   
   - \frac{\gamma \sqrt{-1}}{2} \bpartial \phi \wedge \partial \omega^{p-1}
 \end{equation}
  and
\[ \Omega_{\phi} = X [\phi] + \sqrt{-1} \partial \bpartial  \phi \wedge \omega^{p-1}. \]
Then 
\begin{equation}
\label{gg-I10abcc}
 \partial \bpartial \Omega_{\phi} = \partial \bpartial \Omega + (1 - \gamma) \sqrt{-1} 
  \partial \bar\partial \phi \wedge \partial \bar\partial \omega^{p-1}. 
\end{equation}  
When $\gamma = 1$, therefore, $\Omega_{\phi} - \Omega$ is pluriclosed and 
$[\Omega_{\phi} - \Omega] = 0$ in $H^{p,p}_{A} (M, \bfR)$
as
\begin{equation}
\label{gg-I10abcd}
 \Omega_{\phi} - \Omega =   \partial \alpha + \bar\partial \bar\alpha, \;\; 
  \alpha = \frac{1}{2} \sqrt{-1} \bar \partial \phi \wedge \omega^{p-1}. 
  \end{equation} 
 In particular, if $\Omega$ is pluriclosed so is $\Omega_{\phi}$ and 
$[\Omega_{\phi}] = [\Omega] \in H^{p,p}_{A} (M, \bfR)$. 
When $p = n-1$ and $\Omega = \omega_0^{n-1}$ where $\omega_0$ is a Gauduchon metric, this fact was used in \cite{STW17} in the proof of Gauduchon conjecture.
In equation~\eqref{I1} if we replace a pluriclosed $(p,p)$ form $\Omega$ by $X [\phi]$ defined in \eqref{gg-I10ab} for $\gamma = 1$, we see from \eqref{gg-I10abcc} and \eqref{gg-I10abcd} that it preserves the corresponding Aeppli cohomology class of $\Omega$.

We first state a result which is a special case of our main theorem.
Let $\sigma_k$ be the $k$-th elementary symmetric function
\[ \sigma_k (\lambda) = \sum_{i_1 < \cdots < i_k} \lambda_{i_1} \cdots \lambda_{i_k},
\;\; 1 \leq k \leq N  \]
defined in the Garding cone 
\[ \Gamma_k = \{\lambda \in \bfR^N: \sigma_j (\lambda) > 0,\;
\mbox{$\forall$ $1 \leq j \leq k$}\}. \]

\begin{theorem}
\label{theorem-I1k}
Let $\phi_0 \in C^{\infty} (M)$ be a $\Gamma_k$-admissible function:
\[ \Lambda (X [\phi_0] + \sqrt{-1} \partial \bpartial \phi_0 \wedge \omega^{p-1}) \in \Gamma_k 
\;\; \mbox{on $M$} \]
where $X [\phi]$ is defined in \eqref{gg-I10ab}. 
For $k \leq {p N}/{n}$, the equation
\begin{equation}
\label{I1k}
\begin{aligned}
  \frac{\partial \phi}{\partial t} 
      = \,& \sigma_k^{\frac{1}{k}} (\Lambda(X [\phi] + \sqrt{-1}\partial \bpartial \phi \wedge \omega^{p-1})) - \psi
    \end{aligned}
\end{equation}
admits an admissible solution $\phi \in C^{\infty} (M \times \{t > 0\})$ with $\phi |_{t=0} = \phi_0$,
and the normalized function 
\begin{equation}
\label{I2}
    \tilde{\phi} := \phi - \frac{\int_M \phi \omega^n}{\int_M \omega^n}
\end{equation}
converges uniformly to an admissible solution 
$\varphi \in C^{\infty} (M)$ of the elliptic equation
\begin{equation}
\label{I4k}
    \sigma_k^{\frac{1}{k}} (\Lambda (X [\varphi] + \sqrt{-1} \partial \bpartial \varphi \wedge \omega^{p-1})) 
  = \psi + b
\end{equation}
for some constant $b$ as $t \to \infty$.
\end{theorem}

\begin{remark}

For $p = n-1$, in particular, Theorem~\ref{theorem-I1k} holds for $1 \leq k \leq n-1$. 
\end{remark}

From both the PDE point of view and potential geometric applications it would be important to 
consider more general (nonlinear) dependence of $X [\phi]$ on $\partial \phi$ and $\bar \partial \phi$.
As in the general PDE theory, we need to impose growth rate conditions on the gradient dependence.
For simplicity we assume 
\[ X [\phi] = X (z, \partial \phi, \bar \partial \phi), \;\; 
   \psi [\phi] = \psi (z, \partial \phi, \bar \partial \phi), \]
which are independent of $\phi$, to satisfy
\begin{equation}
\label{G0.1a}
 \begin{aligned}
   & |D_{\zeta} X (z, \zeta,\bar {\zeta})| \leq \varrho_0 |\zeta| + \varrho_1, \;\;
     |D_{\zeta} \psi (z, \zeta,\bar {\zeta})| 
       \leq  \varrho_0 f (|\zeta|^2 \mathbf{1})/|\zeta| + \varrho_1
        \end{aligned}
\end{equation}
 and
\begin{equation}
\label{G0.5a}
    |\nabla_z X | \leq |\zeta| \left(\varrho_0 f (|\zeta|^2 \mathbf{1}) 
    + \varrho_1\right), \;\; 
    |\nabla_z \psi| \leq  |\zeta| \left(\varrho_0 |\zeta|^2+\varrho_1\right),  
\end{equation}
where 
 $\varrho_0 = \varrho_0 (z, |\zeta|) \to 0^+$ as $|\zeta| \to \infty$, and
 $\varrho_1 > 0$ is a constant; these conditions will only be used in deriving the gradient estimates.

In this paper we are primarily concerned with the long time existence and convergence of admissible 
solutions of the initial value problem~\eqref{I1}-\eqref{gg-I20}.

\begin{theorem}
\label{theorem-I1}
Let $\phi_0 \in C^{\infty} (M)$ be an admissible function with 
\begin{equation}
\label{S0.30}
   c_0 [\phi_0] := \lim_{t \to + \infty} f (t \mathbf{1}) - \sup_M \psi [\phi_0] 
   = \sup_{\Gamma} f - \sup_M \psi [\phi_0] > 0.
\end{equation}
In addition to \eqref{P2}-\eqref{P4}, 
\eqref{G0.1a} and \eqref{G0.5a},
assume
\begin{equation}
\label{S0.1}
    \sum f_i (\Lambda) \Lambda_i \geq - C_0 \sum f_i (\Lambda)
    \text{ in $\Gamma$}
\end{equation}
for some constant $C_0 > 0$,
and the following crucial condition
\begin{equation}
\label{S0.2}
    \text{rank of $\mathcal{C}_{\sigma}^+ \geq \frac{N(n-p)}{n}+1, \; \forall \; \inf\limits_{\Gamma}f\leq \sigma \leq \sup\limits_{\partial\Gamma}f$}
\end{equation}
where $\mathcal{C}_{\sigma}^+$ denotes the tangent cone at infinity of the level 
set $f = \sigma$.
There exists an admissible solution $\phi \in C^{\infty} (M \times \{t > 0\})$ 
of problem~\eqref{I1}-\eqref{gg-I20} for $\Omega = X [\phi]$.
Moreover, the function $\tilde{\phi}$ defined in
\eqref{I2} converges uniformly  to an admissible solution 
$\varphi \in C^{\infty} (M)$ of the elliptic equation
\begin{equation}
\label{I4}
    f (\Lambda (X [\varphi] + \sqrt{-1} \partial \bpartial \varphi \wedge \omega^{p-1})) 
  = \psi[\varphi] + b
\end{equation}
for some constant $b$ as $t \to \infty$.
\end{theorem}

\begin{remark}
See Section~\ref{PR} for the definitions of $\mathcal{C}_{\sigma}^+$ and
 {\em rank} of $\mathcal{C}_{\sigma}^+$ introduced in \cite{Guan14, GN21}; see also \cite{GGQ21}.
\end{remark}

\begin{remark}
Geometrically, condition~\eqref{S0.1} means that the distance to the origin from the tangent plane of the level 
set $f = \sigma$ is bounded (by $\sqrt{n} C_0$). It is satisfied by most of the typical examples.

\end{remark}

\begin{remark}
The function $\tilde{\phi}$ clearly solves the equation
\begin{equation}
\label{I3}
\begin{aligned}
   \frac{\partial \tilde{\phi}}{\partial t} 
   = f (\Lambda (X [\tilde \phi] + \sqrt{-1} \partial \bpartial \tilde \phi \wedge \omega^{p-1}))    
  - \psi [{\tilde \phi}] - \frac{\int_M \frac{\partial \phi}{\partial t} \omega^n}{\int_M \omega^n}.
\end{aligned}
\end{equation}
\end{remark}

\begin{remark}
When $X [\phi] = \chi [\phi] \wedge \omega^{p-1}$ where $\chi [\phi]$ is a real
$(1,1)$ form which may also depend on $\phi$ and its gradient, 
the eigenvalues of $X [\phi] + \sqrt{-1} \partial \bpartial \phi$
have a special form. More precisely, for $I = ({i_1}, \ldots, {i_p}) \in \fI$ (see Section~\ref{PR}), the corresponding eigenvalue is 
\[ \Lambda_I (X [\phi] + \sqrt{-1} \partial \bpartial \phi) 
    = \sum_{i \in I} \lambda_i = \lambda_{i_1} + \cdots + \lambda_{i_p} \]
where 
$\lambda_{1},  \ldots, \lambda_{n}$ are the eigenvalues of 
$\chi [\phi] + \sqrt{-1} \partial \bpartial \phi$ with respect to $\omega$. 
The elliptic counterpart of equation~\eqref{I1} was treated in \cite{GGQ21} in this case. 
It is closely related to the notion of $\bfG$-plurisubharmonicity introduced by R. Harvey and B. Lawson~\cite{HL11, HL12, HL13}, and for $p = n-1$ the form-type Monge-Amp\`ere equation studied by
Fu-Wang-Wu~\cite{FWW11, FWW15} and Sz\'ekelyhidi-Tosatti-Weinkove~\cite{STW17};
see also ~\cite{TW17, TW2019} for earlier results.
\end{remark}

Theorem~\ref{theorem-I1k} is a special case of Theorem~\ref{theorem-I1}.
For $f = \sigma_k^{\frac{1}{k}}$ and $\sigma > 0$, 
the rank of $\mathcal{C}_{\sigma}^+$ is $N-k+1$. This follows from an inequality of Lin-Trudinger~\cite{LT94}.
Consequently, Theorem~\ref{theorem-I1} applies to $f = \sigma_k^{\frac{1}{k}}$ for 
$k \leq p N/n$ but excludes in particular the case $f =  \sigma_N^{\frac{1}{N}}$.

Another interesting example is $f = \log \varrho_k$, $1 \leq k \leq N$ where 
\[ \varrho_k (\Lambda) := \prod_{1 \leq i_1 < \cdots < i_k \leq N}
(\Lambda_{i_1} + \cdots + \Lambda_{i_k})  \]
defined in the cone
\[ \mathcal{P}_k : = \{\Lambda \in \bfR^N:
      \Lambda_{i_1} + \cdots + \Lambda_{i_k} > 0,
         \; \forall \; 1 \leq i_1 < \cdots < i_k \leq N \}. \]
In particular,  $\rho_1 = \sigma_N$ and $\rho_N = \sigma_1$.
For $f = \log \varrho_k$ the rank of $\cC_{\sigma}^+$ is $k$, and Theorem~\ref{theorem-I1} applies for 
$k > (n-p) N /n$.


The organization of the paper is as follows. In Section~\ref{PR} we derive the expression
of equation~\eqref{I1} in local coordinates, and prove the parabolicity under
condition~\eqref{P2}, which is rather nontrivial. 
Sections \ref{S} and \ref{G} are devoted to the second order and gradient estimates respectively. 
In Section~\ref{H} we prove the uniform convergence in Theorem~\ref{theorem-I1}, which needs an extension of Li-Yau's Harnack inequality~\cite{LY86} to equations with lower order terms on Hermitian manifolds. The proof 
presented here is a slightly refined version of that first carried out in \cite{George};  see \cite{Cao1985} and \cite{Gill11} for related results.

\bigskip

\section{The equation in local coordinates and ellipticity} 
\label{PR}
\setcounter{equation}{0}
\medskip

For $1 \leq p \leq n$, denote
\[ \fI = \fI_p = \{(i_1, \ldots, i_p) \in \bfN^p: 1 \leq i_1 < \cdots < i_p \leq n\}  \]
and fix the order in $\fI$: 
\[ I = (i_1, \ldots, i_p) < J = (j_1, \ldots, j_p) \]
provided that $i_l < j_l$ for the first non-equal pair. So 
\[ \fI = \{I_1, \ldots, I_N: I_i < I_j \; \mbox{if $i < j$}\} \]
where 
\[N = \frac{n!}{p! (n-p)!}. \]

In local coordinates $z = (z_1, \ldots, z_n)$ we write
\[ \omega = \sqrt{-1} g_{i\bj} dz_i \wedge d\bz_j \]
and, for $I  = (i_1, \ldots, i_p) \in \fI$,
\[ dz_I = dz_{i_1} \wedge \cdots \wedge dz_{i_p}, \;\;
 d \bz_I = d \bz_{i_1} \wedge \cdots \wedge d \bz_{i_p}. \]
 It follows that 
 \[ \omega^p = p! (\sqrt{-1})^{p^2} \sum \omega_{I\bJ} dz_I \wedge d \bz_J \]
where 
 \[ \omega_{I\bJ} = \det (g_{i_k \bj_l}|_{i_k \in I,  j_l \in J}). \]
 
\begin{definition}
Let $\Omega$ be a $(p,p)$ form given in local coordinates
\[ \Omega =  p! (\sqrt{-1})^{p^2} \Omega_{I\bJ} dz_I \wedge d \bz_J. \]
The {\sl eigenvalues} of $\Omega$ with respect to $\omega$ are defined to be the 
eigenvalues of the matrix $\{\Omega_{I\bJ}\}$ with respect to $\{\omega_{I\bJ}\}$,
i.e. the roots of 
\[ \det (\Omega_{I\bJ} - \lambda \omega_{I\bJ}) = 0. \]
\end{definition}

It is straightforward to verify that this definition is independent of local coordinates.  
Moreover, $\Omega$ is a real form if and only if $\{\Omega_{I\bJ}\}$ is an $N \times N$ Hermitian matrix (see e.g. Demailly~\cite{Demailly}). 
Consequently, the eigenvalues of a real $(p, p)$ form are real numbers. 

We now consider the equation
\begin{equation}
\label{pde2}
    \begin{aligned}
        & \frac{\partial \phi}{\partial t} = f (\Lambda(X [\phi]  + (\sqrt{-1} \partial \bar\partial \phi + \chi) \wedge \omega^{p-1})) - \psi [\phi],  \\
    \end{aligned}
\end{equation}
which is clearly an equivalent form of equation~\eqref{I1},
where $\chi$ is a smooth real $(1, 1)$ form, 
$X [\phi] = X (z, \phi, \partial\phi, \bpartial \phi)$ which is a real 
$(p,p)$ form for given $\phi$, while similarly 
$\psi [\phi] = \psi (z, \phi, \partial\phi, \bpartial \phi)$;
both $X$ and $\psi$ are assumed to be smooth in their variables.

For convenience we shall write
$\fg = \sqrt{-1} \partial \bar\partial \phi + \chi$ and 
\[ Z [\phi] := X[ \phi] + 
\fg \wedge \omega^{p-1}. \]
Define the functions $F$, $G$ by 
\[ F (Z) = G (\fg, X) = f (\Lambda (Z)).  \]
In local coordinates we write 
\[ \fg = \sqrt{-1} \fg_{i\bj} dz_i \wedge d \bz_j, \]
\[ X [\phi] = p! (\sqrt{-1})^{p^2} X_{I\bJ} dz_I \wedge d \bz_J, \]
\[ Z [\phi] = p! (\sqrt{-1})^{p^2} Z_{I\bJ} dz_I \wedge d \bz_J, \]
and 
\[ F^{I\bJ} = \frac{\partial F}{\partial Z_{I\bJ}} (Z), \;\; 
 G^{i\bj} = \frac{\partial G}{\partial \fg_{i\bj}} (\fg).  \]
It is well known from \cite{CNS3} that the matrix $\{F^{I\bJ}\}$ is positive semi-definite if
$\Lambda (Z) \in \Gamma$ under assumption \eqref{P2}.

\begin{lemma}
\label{lemma-P20}
Under condition \eqref{P2}, $\{G^{i\bj}\}$ is positive semi-definite 
for $\Lambda (Z) \in \Gamma$. Consequently, equation~\eqref{pde2} is degenerate parabolic for admissible solutions.
\end{lemma}

\begin{proof}
At a point on $M$ where we assume $g_{i\bj} = \delta_{ij}$, by straightforward calculations,  
\begin{equation}
\label{gg-C2-115}
Z_{I\bJ} =  \left\{ \begin{aligned}
   &  
    X_{I\bI} + \sum_{i \in I} \fg_{i\bi},  
          \;\; I = J, \\
  &   
       X_{I\bJ} +  (-1)^{\alpha^{ij}_{IJ}} 
          {\fg_{i\bj}},  \;\; |I \cap J| = p-1, \; i \in I \setminus J, \; j \in J \setminus I, \\
   & X_{I\bJ},  \;\; |I \cap J| \leq p-2
    \end{aligned} \right. 
\end{equation}    
where $\alpha^{ij}_{IJ} =(i|I)+(j|J)$ and $(i|I)$ denotes the position of $i$ in $I$ so
$(i_k|I) = k$ for $I = (i_1, \ldots, i_p)$. 
It follows that 
\begin{equation}
\label{gg-C2-120}
G^{i\bj} =
\left\{  \begin{aligned}
  & \sum_{I: i \in I} F^{I\bI},  \;\; i = j,\\ 
& \sum_{|I \cap J|= p-1; i \in I \setminus J, j \in J \setminus I} (-1)^{\alpha^{ij}_{IJ}} F^{I\bJ}, 
             \; i \neq j. 
\end{aligned} \right.
\end{equation}
 We need to rewrite it in a different expression. 
 Let 
 $$\fI' = \{(i_1, \ldots, i_{p-1}): 1 \leq i_1 < \cdots < i_{p-1} \leq n\}$$
 and for $I' = \{i_1, \ldots, i_{p-1}\} \in \fI'$, $i \notin I'$, $i_{k-1} < i <i_{k}$, 
 write 
 $$I'_i = \{i_1, \ldots, i_{k-1}, i, i_{k}, \ldots, i_{p-1}\}. $$
 A key observation is the following identity,
 \begin{equation}
 \label{gg-C2-124}
G^{i\bj} = \sum_{I' \in \fI': i, j \notin I'} (-1)^{\alpha^{ij}_{I'_i I'_j}} F^{I'_i\bar{I'}_j}. 
 \end{equation}
For each $I'$ the principal submatrix $\{F^{I'_i\bar{I'}_j}\}$ is positive semidefinite as so is 
$\{F^{I\bar J}\}$. It follows that  $\{G^{i\bj}\}$ is a Hermitian matrix, and for any $ \xi = (\xi_1, \ldots, \xi_n) \in \bfC^n$
 \begin{equation}
 \label{gg-C2-125}
 \begin{aligned}
\sum G^{i\bj} \xi_i \bar \xi_j
    = \,& \sum_{i,j} \sum_{I' \in \fI': i, j \notin I'} (-1)^{\alpha^{ij}_{I'_i I'_j}} F^{I'_i\bar{I'}_j} \xi_i \bar \xi_j \\
    = \,& \sum_{I' \in \fI'} \sum_{i, j \notin I'} F^{I'_i \bar{I'}_j}  [(-1)^{(i | I'_i )} \xi_i]  [(-1)^{(j | I'_j)} \bar \xi_j] 
    \geq 0.
    \end{aligned}
 \end{equation}
This completes the proof.
\end{proof}

It should be noted that the submatrices of form $\{F^{I'_i \bI'_j}\}$ do not include all principal 
$(n-p+1) \times (n-p+1)$ submatrices 
of $\{F^{I\bJ}\}$. 

The following refined inequality is crucial to the second order and gradient estimates derived in 
Sections \ref{S} and \ref{G}.
 
\begin{lemma}
\label{lemma-P10}
In local coordinates, at a point $z_0$ where $g_{i\bj} = \delta_{ij}$, 
assume $f_I \leq f_J$ for $1 \leq I < J \leq N$. Denote 
$\alpha = pN/n$. 
Then 
 \begin{equation}
 \label{gg-C2-126}
\sum G^{i\bj} \xi_i \bar \xi_j
    \geq f_{\alpha} |\xi|^2,
    \;\; \forall \; \xi = (\xi_1, \ldots, \xi_n) \in \bfC^n.
 \end{equation}
\end{lemma}

\begin{proof}
Clearly we may assume $G^{i\bj}$ to be diagonal at $z_0$.
Let $P$ be a unitary matrix diagonalizing $\{F^{I\bJ}\}$
\[ \bar P^T \{F^{I\bJ}\} P = [f_{1}, \ldots, f_N], \]
that is, for $P = \{P^{I\bJ}\}$,
\[  \begin{aligned} 
    F^{I\bJ} 
      = \sum_{K \leq N} f_{K} P^{I\bK} \bar P^{J\bK}.   
\end{aligned} \] 
 By \eqref{gg-C2-124},
 \begin{equation}
 \label{gg-C2-127}
  \begin{aligned}
 \sum G^{i\bi} |\xi_i|^2
    = \,& \sum_{i} \sum_{I' \in \fI': i \notin I'} F^{I'_i\bI'_i} |\xi_i|^2 \\
      = \,& \sum_{i} \sum_{I' \in \fI': i \notin I'} 
              \sum_{K \leq N} f_{K} P^{I'_i\bK} \bar P^{I'_i\bK} |\xi_i|^2 \\
 \geq \,& 
       f_{\alpha} \sum_{i}  |\xi_i|^2 \sum_{I' \in \fI': i \notin I'} 
               \sum_{K \geq {\alpha}}  |P^{I'_i\bK}|^2, 
                 \;\; \forall \; \xi \in \bfC^n.
     \end{aligned}
 \end{equation}
Now we derive \eqref{gg-C2-126} by applying Lemma~\ref{lemma-P30} below for each $1 \leq i \leq n$.
\end{proof}


\begin{lemma}
\label{lemma-P30}
Let $B = \{b^{i\bj}\}$ be an $N \times N$ unitary matrix and $A$ an $\alpha \times \beta$ submatrix 
of $B$ with $\alpha + \beta \geq N +1$. Then
 \begin{equation}
 \label{gg-PR-110}
\sum |b^{i\bj}|^2 \geq 1
 \end{equation}
 where the sum is taken over all entries of $A$. 
\end{lemma}

\begin{proof}
This is an elementary result in linear algebra and can be seen easily as follows. 
It is sufficient to assume $\beta = N - \alpha + 1$.  
Without loss of generality assume 
$A = \{b^{i\bj}\}$ for $1 \leq i \leq \alpha$, $\alpha \leq j \leq N$, and 
\[ \sum_{1 \leq i \leq \alpha, \alpha \leq j \leq n} |b^{i\bj}|^2 = \varepsilon < 1. \]
As $B$ is an unitary matrix, 
\[  \alpha - 1 = \sum_{j = 1}^{\alpha - 1} \sum_{i = 1}^{n} |b^{i\bar j}|^2 
                \geq \sum_{i = 1}^{\alpha} \sum_{j = 1}^{\alpha -1} |b^{i\bar j}|^2   
                     = \alpha - \varepsilon \]
which is a contradiction.
\end{proof}

\begin{remark}
Lemma~\ref{lemma-P30} clearly holds as long as the rows and columns of $B$ consist of
unit vectors. 
\end{remark}

We now briefly recall some definitions and results from \cite{GGQ21}, \cite{Guan14} and \cite{GN21} which are key ingredients in the {\em a priori} estimates and are related to the rank condition~\eqref{S0.2}. 
For 
$\sigma \in (\sup_{\partial \Gamma} f, \sup_{\Gamma} f)$ define
\[ \Gamma^{\sigma} = \{\lambda \in \Gamma: f (\lambda) > \sigma\}. \]

\begin{lemma}[\cite{GGQ21}]
\label{lemma1}
 Under conditions \eqref{P2} and \eqref{P3}, the level hypersurface of $f$
 \[ \partial \Gamma^{\sigma} = \{\lambda \in \Gamma: f (\lambda) = \sigma\}, \]
which is the boundary of $\Gamma^{\sigma}$, is smooth and convex. 
\end{lemma}

This is clearly true with strict inequality in \eqref{P2}, but still remains valid under the slightly weaker hypothesis.

For $\lambda \in \partial \Gamma^{\sigma}$ define
\[ \nu_{\lambda} = \frac{Df (\lambda)}{|Df (\lambda)|} \]
which is the unit normal vector to $\partial \Gamma^{\sigma}$ at $\lambda$. 

\begin{definition}[\cite{Guan14}] 
For $\mu \in \bfR^n$
let 
 \[ S^{\sigma}_{\mu} = \{\lambda \in \partial
\Gamma^{\sigma}: \nu_{\lambda} \cdot (\mu - \lambda) \leq 0\}. \]
The  {\em tangent cone at infinity} 
to $\Gamma^{\sigma}$ is defined as
\[ \begin{aligned}
\cC^+_{\sigma}
 \,& = \{\mu \in \bfR^n:
              S^{\sigma}_{\mu} \; \mbox{is compact}\}.
    \end{aligned} \]
\end{definition}

Clearly $\cC^+_{\sigma}$ is a symmetric convex cone.

\begin{theorem}[\cite{Guan14}]
\label{Guan14-thm10}
{\bf a)} $\cC^+_{\sigma}$ is open. 
{\bf b)} For any compact subset $K$ of $\cC^+_{\sigma}$
there exist constants $\epsilon, \; R > 0$ such that 
\begin{equation}
\label{Guna14-10}
    \sum f_{\lambda_i} (\mu_i - \lambda_i) \geq \epsilon \sum f_{\lambda_i} + \epsilon,
    \;\; \forall \; \mu \in K, \; \lambda \in \partial \Gamma^{\sigma}, \; |\lambda| \geq R. 
\end{equation}
\end{theorem}

By \eqref{P2} we see that the normal vector of any supporting hyperplane to 
${\mathcal{C}}_{\sigma}^+$ belongs to 
$\ol \Gamma_N = \{\lambda \in \bfR^N: \lambda_i \geq 0\}$. 

\begin{definition}[\cite{GN21}]
The {\em rank of ${\mathcal{C}}_{\sigma}^+$} is defined to be
\[ \min \{r (\nu): \mbox{$\nu$ is the unit normal vector of a supporting plane
to ${\mathcal{C}}_{\sigma}^+$}\} \]
where $r (\nu)$ denotes the number of non-zero
components of $\nu$.
\end{definition}

\begin{lemma}[\cite{GGQ21}]
\label{lemma 3} 
Assume \eqref{P3}, \eqref{P2} and \eqref{S0.1} hold. 
Let $r$ be the rank of $\mathcal{C}_{\sigma}^+$.
 There exists $c_0 > 0$ such that at any point 
 $\lambda \in \partial \Gamma^{\sigma}$ where without loss of generality we assume 
 $f_{1} \leq \cdots \leq f_{N}$, 
\[ \sum_{i \leq N-r+1} f_i \geq c_0 \sum f_i. \]
\end{lemma}

\bigskip

\section{Second Order Estimates}
\label{S}
\setcounter{equation}{0}

\medskip


Throughout this article $\nabla$ denotes the Chern connection of the metric $\omega$, so $\nabla$ is the unique connection on $TM$ satisfying 
\[ d \hspace{1pt}\langle s_1,s_2\rangle=\langle\nabla s_1, s_2\rangle+\langle s_1,\nabla s_2\rangle \text{ and } \nabla^{0,1}=\bpartial \] 
for any smooth sections $s_1$ and $s_2$ of $TM$, where $\nabla^{0,1}$ denotes the projection of $\nabla$ to $T^{0,1}M$. 
In local coordinates $\nabla$ is determined by 
\[ \nabla_i \partial_j =\Gamma^{k}_{ij} \partial_k \text{ where } \Gamma_{ij}^k=g^{k\bl}\partial_ig_{j\bl}, \]
while the torsion and curvature tensors are given by
\[ T_{ij}^l=\Gamma^l_{ij}-\Gamma^l_{ji} \]
and
\[ R_{i\bj k\bl }=-\partial_{\bj} \partial_i g_{k\bl}+g^{p\bq}\partial_ig_{k\bq}\partial_{\bj}g_{p\bl}, \]
respectively.

We establish second order estimates for admissible solutions of 
equation~\eqref{pde2} in this section.

\begin{theorem}
\label{theorem-I20}
Let $\phi \in C^{4,1} (M \times (0,T])$ be an admissible solution of 
equation~\eqref{pde2} where $T > 0$.
Assume conditions \eqref{P2}-\eqref{P4}, ~\eqref{S0.1} and \eqref{S0.2} hold. 
There exists constant $C$ depending on $|\phi|_{C^1 (M \times [0, T])}$, $T$ and
\begin{equation}
\label{S0.3}
   c_0 [\phi] := \lim_{t \to + \infty} f (t \mathbf{1}) - \sup_M \psi [\phi] 
   = \sup_{\Gamma} f - \sup_M \psi [\phi] > 0,
\end{equation}
where ${\bf 1} = (1, \ldots, 1) \in \Gamma \subset \bfR^N$,  as well as geometric quantities of $(M, \omega)$ such that
\[ \sup_{M \times [0,T)} |\partial \bpartial \phi|_{\omega} \leq C. \]
\end{theorem}

It is clearly enough to derive an upper bound for $\fg = \chi + \sqrt{-1} \partial  \bar \partial \phi$. 
We follow \cite{GN21} to use an idea of Tossati-Weinkove \cite{TW17} and consider
the quantity which is given in local coordinates 
\begin{equation}\label{S1}
\begin{aligned}
A:= \sup\limits_{(z,t)\in M\times [0,T)}\max\limits_{\xi\in T_z^{1,0}M}e^{(1+\gamma)\eta}\fg_{p\bq}\xi_p\overline{\xi}_q(g^{k\bl}\fg_{i\bl}\fg_{k\bj}\xi_i\overline{\xi}_j)^{\frac{\gamma}{2}}/|\xi|^{2+\gamma}
\end{aligned}
\end{equation}
where $\eta$ is a function depending on $|\nabla\phi|$ and $\gamma >0$ is a small constant to be chosen.
Assume that $A$ is achieved at a point $(z_0, t_0) \in M \times [0,T)$ for some 
$\xi\in T_{z_0}^{1,0}M$. We choose local coordinates around $z_0$ such that 
$g_{i\bj}=\delta_{ij}$ and $T_{ij}^k=2\Gamma^k_{ij}$ using a lemma of Streets and Tian \cite{ST11}, and that $\fg_{i\bj}$ is diagonal at $z_0$ with 
$\fg_{1\bar{1}} \geq \cdots \geq \fg_{n\bn}$. We shall assume 
$\fg_{1\bar{1}} \geq 1$; otherwise we are done.

By \cite{TW17} (see also \cite{GN21}) $\xi=\partial_1$ at $(z_0,t_0)$
for $\gamma$ sufficiently small. 
Let $W = g_{1\bar{1}}^{-1} g^{k\bl} \fg_{1\bl} \fg_{k\bar{1}}$.
We follow \cite{GN21} to calculate at $(z_0,t_0)$ where
$W = \fg_{1\bar{1}}^2$,
\begin{equation}
\label{S2}
 \begin{aligned}
\frac{\partial_i (g_{1\bar{1}}^{-1} \fg_{1\bar{1}})}{\fg_{1\bar{1}}}
   +  \frac{\gamma \partial_i W}{2 W} + (1+ \gamma)  \partial_i \eta = \,& 0, 
   \;\; 1 \leq i \leq n 
\end{aligned}
\end{equation}
\begin{equation}
\label{S2t}
 \begin{aligned}
\frac{\partial_t (g_{1\bar{1}}^{-1} \fg_{1\bar{1}})}{\fg_{1\bar{1}}}
   +  \frac{\gamma \partial_t W}{2 W} + (1+ \gamma)  \partial_t \eta \geq \,& 0 
\end{aligned}
\end{equation}
and
\begin{equation}
\label{S3}
\begin{aligned}
0 \geq \,&
   \frac{1}{\fg_{1\bar{1}}} G^{i\bj} \bpartial_j \partial_i (g_{1\bar{1}}^{-1} \fg_{1\bar{1}})
   - \frac{1}{\fg_{1\bar{1}}^2} G^{i\bj} \partial_i (g_{1\bar{1}}^{-1} \fg_{1\bar{1}}) 
        \bpartial_j (g_{1\bar{1}}^{-1} \fg_{1\bar{1}}) \\
  & + \frac{\gamma}{2 W} G^{i\bj} \bpartial_j \partial_i W
     - \frac{\gamma}{2 W^2} G^{i\bj} \partial_i W \bpartial_j W
   + (1 + \gamma) G^{i\bj} \bpartial_j \partial_i \eta.
\end{aligned}
\end{equation}

We make use of the following identities derived in \cite{GN21},
\begin{equation}
\label{S4}
\begin{aligned}
\partial_i (g_{1\bar{1}}^{-1} \fg_{1\bar{1}}) = \nabla_i \fg_{1\bar{1}}, \;\;
 \partial_i W 
 = 2 \fg_{1\bar{1}} \nabla_i \fg_{1\bar{1}},
     \end{aligned}
      \end{equation}
\begin{equation}
\label{S5}
 \begin{aligned}
\bpartial_j  \partial_i (g_{1\bar{1}}^{-1} \fg_{1\bar{1}})
       = \,& \nabla_{\bj} \nabla_i \fg_{1\bar{1}}
    + (\ol{\Gamma_{j1}^m} \nabla_i \fg_{1\bm}
        - \ol{\Gamma_{j1}^1} \nabla_i \fg_{1\bar{1}}) \\
      & + (\Gamma_{i1}^m \nabla_{\bj} \fg_{m\bar{1}}
            - \Gamma_{i1}^1 \nabla_{\bj} \fg_{1\bar{1}})
      + (\Gamma_{i1}^1  \ol{\Gamma_{j1}^1} - \Gamma_{i1}^m \ol{\Gamma_{j1}^m}) \fg_{1\bar{1}}.
     \end{aligned}  
 \end{equation}     
      and
\begin{equation}
\label{S6}
 \begin{aligned}
\bpartial_j \partial_i W
   = \,& 2 \fg_{1\bar{1}} \nabla_{\bj} \nabla_i \fg_{1\bar{1}}
   + 2 \nabla_i \fg_{1\bar{1}} \nabla_{\bj} \fg_{1\bar{1}}
   + \sum_{l>1} \nabla_i \fg_{l\bar{1}} \nabla_{\bj} \fg_{1\bl}    \\
     &  + \sum_{l > 1} (\nabla_i \fg_{1\bl} + {\Gamma_{i1}^l} \fg_{l\bl})
       (\nabla_{\bj} \fg_{l\bar{1}} + \ol{\Gamma_{j1}^l} \fg_{l\bl}) \\
   & + \fg_{1\bar{1}} \sum_{l > 1} (\ol{\Gamma_{j1}^l} \nabla_i \fg_{1\bl}
       + \Gamma_{i1}^l  \nabla_{\bj} \fg_{l\bar{1}}) \\
  &    -  \fg_{1\bar{1}} \sum_{l > 1} \Gamma_{i1}^m \ol{\Gamma_{j1}^m} (\fg_{1\bar{1}} +  \fg_{l\bl)}.
              \end{aligned}
 \end{equation}
Therefore, 
\begin{equation}
\label{S7}
\begin{aligned}
G^{i\bj} \partial_i W \bpartial_j W
= 4 \fg_{1\bar{1}}^2 G^{i\bi} \nabla_i \fg_{1\bar{1}} \nabla_{\bj} \fg_{1\bar{1}},
\end{aligned}
\end{equation}
\begin{equation}
\label{S8}
\begin{aligned}
G^{i\bj} \partial_i (g_{1\bar{1}}^{-1} \fg_{1\bar{1}}) \bpartial_j (g_{1\bar{1}}^{-1} \fg_{1\bar{1}})
= G^{i\bj} \nabla_i \fg_{1\bar{1}} \nabla_{\bj} \fg_{1\bar{1}}
\end{aligned}
\end{equation}
and by Cauchy-Schwarz inequality,
\begin{equation}
\label{S9}
 \begin{aligned}
G^{i\bj} \bpartial_j \partial_i W
\geq \,& 2 \fg_{1\bar{1}} G^{i\bj} \nabla_{\bj} \nabla_i \fg_{1\bar{1}}
    + 2 G^{i\bj}  \nabla_i \fg_{1\bar{1}} \nabla_{\bj} \fg_{1\bar{1}} \\
  + \,& \sum_{l > 1} G^{i\bj} \nabla_i \fg_{1\bl} \nabla_{\bi} \fg_{l\bar{1}}
  + \frac{1}{2} \sum_{l > 1} G^{i\bj} \nabla_i \fg_{1\bl} \nabla_{\bj} \fg_{l\bar{1}}
   - C \fg_{1\bar{1}}^2 \sum G^{i\bi},
 \end{aligned}
 \end{equation}
\begin{equation}
\label{S10}
 \begin{aligned}
G^{i\bj} \bpartial_j  \partial_i (g_{1\bar{1}}^{-1} \fg_{1\bar{1}})
\geq \,& G^{i\bj} \nabla_{\bj} \nabla_i \fg_{1\bar{1}}
  - \frac{\gamma}{8 \fg_{1\bar{1}}} \sum_{l > 1} G^{i\bj}
            \nabla_i \fg_{1\bl} \nabla_{\bj} \fg_{l\bar{1}}
        - C \fg_{1\bar{1}} \sum G^{i\bi}.
 \end{aligned}
 \end{equation}

By \eqref{S2}, \eqref{S2t} and \eqref{S4}, 
  \begin{equation}
\label{S11}
 \begin{aligned}
\nabla_i \fg_{1\bar{1}} + \fg_{1\bar{1}} \partial_i \eta = 0, \;\;
\nabla_{\bi} \fg_{1\bar{1}} + \fg_{1\bar{1}} \bpartial_i \eta = 0,
\end{aligned}
\end{equation}
 \begin{equation}
\label{S11t}
 \begin{aligned}
\partial_t \fg_{1\bar{1}} + \fg_{1\bar{1}} \partial_t \eta \geq 0, 
\end{aligned}
\end{equation}
while by \eqref{S3}-\eqref{S11},
 \begin{equation}
\label{S12}
\begin{aligned}
0  \geq \,& \frac{1}{\fg_{1\bar{1}}}  G^{i\bj} \nabla_{\bj} \nabla_i \fg_{1\bar{1}}
     - G^{i\bj}  \nabla_i \eta \nabla_{\bj} \eta + G^{i\bj} \bpartial_j \partial_i \eta \\
 & + \frac{\gamma}{\fg_{1\bar{1}}^2} \sum_{l > 1} 
         G^{i\bj} \nabla_i \fg_{1\bl} \nabla_{\bj} \fg_{l\bar{1}} 
    + \frac{\gamma}{16 \fg_{1\bar{1}}^2} \sum_{l > 1} G^{i\bj}
            \nabla_i \fg_{1\bl} \nabla_{\bj} \fg_{l\bar{1}}  - C \sum G^{i\bi}.
\end{aligned}
\end{equation}
Using the formula
\begin{equation}
\label{S13}
 \begin{aligned}
 \nabla_{\bj} \nabla_i \fg_{1\bar{1}}  -  \nabla_{\bar{1}} \nabla_1 \fg_{i\bj}
   = \,&  R_{i\bj1\bar{1}} \fg_{1\bar{1}} - R_{1\bar{1} i\bj} \fg_{i\bj}
         - T_{i1}^l \nabla_{\bj} \fg_{l\bar{1}}  - \ol{T_{j1}^l} \nabla_i \fg_{1\bl} \\
     &  - T_{i1}^l  \ol{T_{j1}^l} \fg_{l\bl} + H_{i\bj}
  \end{aligned}
 \end{equation}
we obtain 
\begin{equation}
\label{S14}
 \begin{aligned}
G^{i\bj} \nabla_{\bj} \nabla_i \fg_{1\bar{1}}
   \geq \,& G^{i\bj} \nabla_{\bar1} \nabla_1\fg_{i\bj}
         -  \frac{\gamma}{32 \fg_{1\bar{1}}} 
             \sum\limits_{l>1}G^{i\bj}  \nabla_i \fg_{1\bl} \nabla_{\bj} \fg_{l\bar{1}} \\
       &  - C \fg_{1\bar{1}} \sum G^{i\bi} + G^{i\bj} H_{i\bj};
 \end{aligned}
 \end{equation}
 see e.g. \cite{GN21}, where
 \[ \begin{aligned}
    H_{i\bj} = \,& \nabla_{\bj} \nabla_i \chi_{1\bar{1}}
                       -  \nabla_{\bar{1}} \nabla_1 \chi_{i\bj}
      - T_{i1}^l \nabla_{\bj} \chi_{l\bar{1}} - \ol{T_{j1}^l} \nabla_i \chi_{1\bl} \\
    &  + R_{i\bj 1\bl} \chi_{l\bar{1}} - R_{1\bar{1} i\bl} \chi_{l\bj}
       - T_{i1}^k  \ol{T_{j1}^l} \chi_{k\bl}.
  \end{aligned} \]

We now make use of equation~\eqref{pde2} which we rewrite as
\begin{equation}
 \label{gg-C2-128}
        \frac{\partial \phi}{\partial t} = F (Z_{I\bJ}) - \psi[\phi].
 \end{equation}
Differentiate \eqref{gg-C2-128} twice to obtain
 \begin{equation}
\label{gg-C2-130} 
     \partial_k \phi_t =  F^{I\bJ} \nabla_k Z_{I\bJ} - \partial_k \psi [\phi]
\end{equation}
and, by the concavity of $f$,
\begin{equation}
\label{gg-C2-140}
  \begin{aligned}
     \bar \partial_{k} \partial_{k} \phi_t 
      = \,& F^{I\bJ} \nabla_{\bk} \nabla_k Z_{I\bJ} 
              + F^{I\bJ, K\bL} \nabla_k Z_{I\bJ} \nabla_{\bk} Z_{K\bL} 
              - \bar \partial_{k} \partial_k \psi [\phi] \\
 \leq \,& F^{I\bJ} \nabla_{\bk} \nabla_k Z_{I\bJ} 
             - \bar \partial_{k} \partial_k \psi [\phi].
      \end{aligned}
\end{equation}

Next, 
\begin{equation}
\label{S19}
\partial_k \psi [\phi] = \psi_k + \psi_{\phi} \partial_k \phi 
        + \psi_{\zeta_{\alpha}} \partial_k \partial_{\alpha} \phi 
        + \psi_{\bzeta_\alpha} \partial_k \bpartial_{\alpha} \phi
        \end{equation}  
and, by \eqref{S11} and \eqref{S11t}, 
\begin{equation}
\label{S21.5}
    \begin{aligned}
 \bar \partial_{1} \partial_1 \phi_t
      = \,& \partial_t (\fg_{1\bar 1} - \chi_{1\bar 1}) \\
      = \,& \partial_t \fg_{1\bar 1} - \chi_{1\bar{1},\phi} \phi_t 
                - 2 \fRe\{\chi_{1\bar{1}, \zeta_{\alpha}} \partial_{\alpha} \phi_t\},
    \end{aligned}
\end{equation} 
\begin{equation}
\label{S20}
\begin{aligned}
  \bar \partial_1 \partial_1 \psi [\phi]
     \geq \,& 2 \fRe\{\psi_{\zeta_\alpha} \nabla_{\bar{1}} \nabla_1 \nabla_{\alpha} \phi\}
                  -  C |A_1|^2 \\   
     \geq \,& 2 \fRe\{\psi_{\zeta_\alpha} \nabla_{\alpha} \fg_{1\bar{1}}\}
                  - C |A|^2 \\
      =  \,& - 2 \fg_{1\bar{1}} 
                 \fRe\{\psi_{\zeta_\alpha} \partial_{\alpha} \eta\}
                - C |A|^2 
\end{aligned}
\end{equation}   
where and in the sequel we denote
\[ |A_i|^2 = \fg_{i\bi}^2 + \sum_{k} |\nabla_i \nabla_k \phi|^2, \;\; 
    |A|^2 =  \sum |A_i|^2. \]
Similarly, 
\begin{equation}
\label{gg-C2-160'}
 \begin{aligned}
F^{I\bJ} \nabla_{\bar 1} \nabla_1 X_{I\bJ} 
   \leq \,&  2 \fRe\{F^{I\bJ} X_{I\bJ, \zeta_{\alpha}} \nabla_{\alpha} \fg_{1 \bar{1}}\}
                    + C |A|^2 \sum G^{i\bi}  \\
       =  \,& -  2 \fg_{1\bar{1}}  
                   \fRe\{F^{I\bJ} X_{I\bJ, \zeta_{\alpha}} \partial_{\alpha} \eta\}
                    + C |A|^2 \sum G^{i\bi}.
             \end{aligned}
 \end{equation}  
It follow from \eqref{gg-C2-125}, 
\eqref{gg-C2-140} and \eqref{S21.5}-
\eqref{gg-C2-160'} that  
\begin{equation}
 \label{gg-C2-125'}
 \begin{aligned}
G^{i\bj} \nabla_{\bar 1} \nabla_1 \fg_{i\bar{j}} 
       = \,& \sum_{I' \in \fI': i, j \notin I'} F^{I'_i\bar{I'}_j} \nabla_{\bar 1} 
                \nabla_1 \fg_{i\bj} \\
       = \,& \sum_{I' \in \fI': i, j \notin I'} F^{I'_i\bar{I'}_j} 
             \nabla_{\bar 1} \nabla_1(Z_{I'_i I'_\bj} - X_{I'_i I'_\bj}) \\
      = \,& \sum_{I, J \in \fI} F^{I\bar{J}} \nabla_{\bar 1} \nabla_1(Z_{I\bJ} - X_{I\bJ}) \\
  \geq \,& \bar \partial_{1} \partial_1  (\phi_t + \psi) - F^{I\bar{J}} 
               \nabla_{\bar 1} \nabla_1 X_{I\bJ}  \\
 \geq  \,& - \fg_{1\bar{1}}  \eta_t - 2 \fg_{1\bar{1}} 
                \fRe\{(\psi_{\zeta_\alpha} - F^{I\bJ} X_{I\bJ, \zeta_{\alpha}} )
                \partial_{\alpha} \eta\} \\
              &  - \chi_{1\bar{1},\phi} \phi_t 
                 - 2 \fRe\{\chi_{1\bar{1}, \zeta_{\alpha}} \partial_{\alpha} \phi_t\} - C |A|^2 \sum G^{i\bi}.
  \end{aligned}
 \end{equation}
 We shall also need
 \begin{equation}
 \label{gg-C2-125'1}
 \begin{aligned}
   \partial_k  \phi_t - G^{i\bj} \nabla_k \fg_{i\bj} 
       = \,& F^{I\bar{J}} \nabla_k X_{I\bJ} -  \nabla_k \psi 
   \leq  C |A| \sum G^{i\bi} + C |A|.
  \end{aligned}
 \end{equation}

If $\chi$ is simply a smooth $(1,1)$ form independent of $\phi$ and its gradient, then
\begin{equation}
\label{S18'}
G^{i\bj} H_{i \bj} \geq -  C \sum G^{i\bi}.
    \end{equation} 
In general for $\chi = \chi (\partial \phi, \bar \partial \phi, \phi, z)$, 
a similar calculation using \eqref{S11} and \eqref{gg-C2-130} 
 yields 

\begin{equation}
\label{S18}
\begin{aligned}
G^{i\bj} H_{i \bj} 
   \geq \,&  2 \fRe\{G^{i\bj} \chi_{1 \bar{1}, \zeta_{\alpha}}
                     \nabla_{\alpha} \fg_{i\bj}\} 
                -  2 \fg_{1\bar{1}} \fRe\{G^{i\bj}  \chi_{i \bj, \zeta_{\alpha}} \nabla_{\alpha} \eta\}
                    - C |A|^2 \sum G^{i\bi}.  \\
 \geq \,& 2 \fRe\{\chi_{1 \bar{1}, \zeta_{\alpha}}  \partial_{\alpha} (\psi+\phi_t)\}
     + 2 \fg_{1\bar{1}} \fRe\{G^{i\bj}  \chi_{i \bj, \zeta_{\alpha}} \partial_{\alpha} \eta\} 
        - C |A|^2 \sum G^{i\bi} \\
 \geq \,& 2 \fg_{1\bar{1}} \fRe\{G^{i\bj}  \chi_{i \bj, \zeta_{\alpha}} \partial_{\alpha} \eta\} 
             +  2 \fRe\{\chi_{1\bar{1}, \zeta_{\alpha}} \partial_{\alpha} \phi_t\}
             - C |A|^2 \sum G^{i\bi} - C |A|.        
             \end{aligned}
 \end{equation}  

Let $\cL$ denote the linearized operator
\begin{equation}
\label{S181}
 \cL = \partial_t - G^{i\bj} \bpartial_j \partial_i - B^{\alpha} \partial_{\alpha}
- \bar B^{\alpha} \bar \partial_{\alpha} 
\end{equation}
where 
\[          B^{\alpha} = F^{I\bJ} X_{I\bJ, \zeta_{\alpha}} 
                + G^{i\bj}  \chi_{i \bj, \zeta_{\alpha}} -  \psi_{\zeta_\alpha}.           \]
We derive from \eqref{S12}, \eqref{S14}, \eqref{gg-C2-125'} and \eqref{S18},
\begin{equation}
\label{S12'}
\begin{aligned}
- \fg_{1\bar{1}} \cL \eta  
     \leq \,& - \fg_{1\bar{1}} G^{i\bj}  \partial_i \eta \bar \partial_{j} \eta
                  + C |A|^2 \sum G^{i\bi} + C |A|.
\end{aligned}
\end{equation}

 Let $\eta=-\log{h}$, where $h=1-\gamma|\nabla \phi|^2$ and require that 
 $\gamma > 0$ be sufficiently small to satisfy $2 \gamma |\nabla \phi|^2 \leq 1$. 
 By straightforward calculations using \eqref{gg-C2-125'1},
\begin{equation}
\label{S27}
    \begin{aligned}
 \partial_t \eta = \frac{\gamma}{h} \partial_t |\nabla \phi|^2 
     =  \frac{2 \gamma}{h} \fRe\{\bar \partial_k \phi (G^{i\bj} \nabla_k \fg_{i\bj} 
             - F^{I\bJ} \nabla_k X_{I\bJ} - \partial_{k} \psi)\},
    \end{aligned}
\end{equation}
\begin{equation} \label{S23}
 \partial_i  \eta = \frac{\gamma}{h} \partial_i  |\nabla \phi|^2 
       = \frac{\gamma}{h}  (\nabla_k \phi \nabla_i \nabla_{\bk} \phi 
      + \nabla_{\bk} \phi \nabla_i \nabla_k \phi)
\end{equation}
and
\begin{equation}
\label{S24}
  \begin{aligned}
G^{i\bj} (\bpartial_j \partial_i\eta - \partial_i \eta \bar \partial_j \eta)
=  \,& \frac{\gamma}{h} G^{i\bj} \bpartial_j \partial_i |\nabla \phi|^2 \\
     \geq  \,&  \frac{\gamma (1 - \gamma)}{h}  G^{i\bi} \fg_{i\bi}^2 
                   + \frac{\gamma}{h} G^{i\bj} \nabla_i \nabla_{k} \phi \nabla_{\bj} \nabla_{\bk} \phi \\
      &    + \frac{2 \gamma}{h} \fRe\{G^{i\bj} \nabla_{\bk} \phi 
      \nabla_k (\fg_{i\bj} - \chi_{i\bj})\}
          - C |A| \sum G^{i\bi}.
\end{aligned}
\end{equation}
From \eqref{S27}-\eqref{S24} we derive 

\begin{equation}
\label{S12'a}
\begin{aligned}
-  \cL \eta + G^{i\bj}  \partial_i \eta \bar \partial_{j} \eta
  \geq \,& \frac{\gamma (1 - \gamma)}{h}  G^{i\bi} \fg_{i\bi}^2 + \frac{\gamma}{h} 
                 G^{i\bj} \nabla_i \nabla_{k} \phi \nabla_{\bj} \nabla_{\bk} \phi \\
           & - C |A| \sum G^{i\bi} - C |A|.
\end{aligned}
\end{equation}
Fixing $\gamma$ sufficiently small, we obtain 
 \begin{equation}
\label{S29}
\begin{aligned}
G^{i\bi} \fg_{i\bi}^2 + G^{i\bj} \nabla_i \nabla_{k} \phi \nabla_{\bj} \nabla_{\bk} \phi
     \leq \frac{C |A|^2}{\fg_{1\bar{1}}}   \sum G^{i\bi} + C |A|.
\end{aligned}
\end{equation}  

Applying Lemma~\ref{lemma-P10} and Lemma~\ref{lemma 3} we derive 
  \begin{equation}
 \label{S31}
\begin{aligned}
G^{i\bi} \fg_{i\bi}^2 + G^{i\bj} \nabla_i \nabla_{k} \phi \nabla_{\bj} \nabla_{\bk} \phi
\geq c_0 |A|^2 \sum  G^{i\bi}.         
 \end{aligned}
\end{equation}  
By the concavity of $f$,
\[  \begin{aligned}
    \sqrt{\fg_{1\bar1}}  \sum  f_{\Lambda_I}  
        = \,& \sqrt{\fg_{1\bar1}} \sum  f_{\Lambda_I}  
        - \sum  f_{\Lambda_I} \Lambda_I + \sum  f_{\Lambda_I} \Lambda_I \\
   \geq \,&  f (\sqrt{\fg_{1\bar1}} {\bf{1}}) - f (\Lambda (\fg)) 
        - \sqrt{\fg_{1\bar1}} \sum f_{\Lambda_I} 
       - \frac{4}{\sqrt{\fg_{1\bar1}}} \sum  f_{\Lambda_I} \Lambda_I^2.
  \end{aligned} \]
Therefore, when $\fg_{1\bar1}$ is sufficiently large we obtain by assumption~\eqref{S0.3},
\[  \sqrt{\fg_{1\bar1}}  \sum  f_{\Lambda_I} 
     + \frac{1}{\sqrt{\fg_{1\bar1}}} \sum  f_{\Lambda_I} \Lambda_I^2
   \geq f (\sqrt{\fg_{1\bar1}} {\bf{1}}) - \psi [\phi]
      \geq c_0 > 0. \]
Combining this with \eqref{S29} and \eqref{S31} we finally derive
 \begin{equation}
 \label{S32}
\begin{aligned}
|A|^{\frac{3}{2}} + |A|^2 \sum G^{i\bi}
     \leq   \frac{C |A|^2}{\sqrt{\fg_{1\bar{1}}}}  \sum G^{i\bi} + C |A|. 
\end{aligned}
\end{equation} 
This gives the upper bound $\fg_{1\bar{1}}\leq C$ depending on $|\phi|_{C^1 (M)}$.

\bigskip

\section{$C^1$ estimates and long time existence} 
\label{G}
\setcounter{equation}{0}

\medskip

Throughout this section let  $\phi \in C^{3,2} (M \times [0,T])$ be an admissible solution of \eqref{pde2} with initial data $\phi|_{t = 0} = \phi_0$. Denote
\[ H_0 = \min_M \frac{\partial \phi}{\partial t} \Big|_{t=0}, \;\; 
   H_1 = \max_M \frac{\partial \phi}{\partial t} \Big|_{t=0}.   \]
We shall assume $H_0 \leq 0 \leq H_1$; otherwise should be replaced by $0$ below. 
Let $\cL$ be the linearized operator at $\phi$.

In local coordinates $\cL$ is given as in \eqref{S181}. For convenience we use the 
following notations
\begin{equation}
\label{GG-E10}
   \begin{aligned}
    B_{\phi} := \,& F^{I\bJ} X_{I\bJ, \phi} + G^{i\bj} \chi_{i\bj, \phi} - \psi_{\phi} \\
    B_k := \,& F^{I\bJ} \nabla_k X_{I\bJ} 
                  + G^{i\bj} \nabla_k \chi_{i\bj} - \nabla_k \psi  \\
  B^{\alpha} := \,& F^{I\bJ} X_{I\bJ, \zeta_{\alpha}} + G^{i\bj} \chi_{i\bj, \zeta_{\alpha}} - \psi_{\zeta_{\alpha}}     
     \end{aligned}
\end{equation}
taking values at $\phi$.

By straightforward calculations as in \eqref{gg-C2-125'},
\begin{equation}
 \label{GG-E50}
 \begin{aligned}
 G^{i\bj} \bar \partial_j \partial_i \partial_t \phi
       = \,& G^{i\bj}  
                  \partial_t (\fg_{i\bj} - \chi_{i\bj} [\phi]) \\
      = \,& \sum F^{I\bar{J}} \partial_t (Z_{I\bJ} - X_{I\bJ}) 
                - G^{i\bj} \partial_t \chi_{i\bj} [\phi] \\
      = \,& \partial_t  (\partial_t \phi + \psi [\phi]) - F^{I\bar{J}} \partial_t X_{I\bJ} [\phi]
        - G^{i\bj} \partial_t \chi_{i\bj} [\phi]. 
  \end{aligned}
 \end{equation}
Therefore, 
\begin{equation}
 \label{GG-E60}
\cL \partial_t \phi - B_{\phi}
    \partial_t \phi = 0. 
 \end{equation}
By the maximum principle we immediately obtain the following estimates.

\begin{lemma}
\label{GG-E-lemma10}
Assume $B_{\phi} \leq K_0$
where $K_0 \geq 0$ is a constant. Then 
\begin{equation}
\label{GG-E30}
   H_0 e^{K_0 t}  \leq \frac{\partial \phi}{\partial t} 
      \leq H_1 e^{K_0 t} \;\;
         \mbox{in $M \times [0,T)$}
\end{equation}
and consequently, 
\begin{equation}
\label{GG-E40}
\min_M \phi_0 + \frac{H_0}{K_0} e^{K_0 t} \leq \phi 
             \leq \max_M \phi_0 + \frac{H_1}{K_0} e^{K_0 t} \;\; 
  \mbox{in $M \times [0,T)$}
\end{equation}
provided $K_0 > 0$; otherwise $|\phi|$ grows at most linearly in $t$. 
\end{lemma}

\begin{proof}
By \eqref{GG-E60} and the assumption $B_{\phi} \leq K_0$,
\[ \cL (e^{- K_0 t} \partial_t \phi) = e^{- K_0 t } (\cL \partial_t \phi - K_0  \partial_t \phi) \leq 0\;\; \mbox{on $M \times [0, T] \cap \{\partial_t \phi \geq 0\}$}. \]
It follows from the maximum principle that 
\[ e^{- K_0 t} \partial_t \phi 
      \leq \max_{M  \times \{t = 0\}} e^{- K_0 t} \partial_t \phi = H_1. \]
Therefore, 
\[ \partial_t \phi  \leq  H_1  e^{K_0 t}. \]
Similarly, 
\[ \cL (e^{- K_0 t} \partial_t \phi) = e^{- K_0 t } (\cL \partial_t \phi - K_0  \partial_t \phi) \geq 0\;\; \mbox{on $M \cap \{\partial_t \phi \leq 0\}$} \]
which gives
\[ \partial_t \phi  \geq  H_0  e^{K_0 t}. \]
Clearly \eqref{GG-E40} follows from \eqref{GG-E30}. 
\end{proof}

To derive gradient estimates it is necessary to impose growth conditions on $X [\phi]$ and
$\psi [\phi]$ regarding their dependence on $\phi$ and $|\partial \phi|$. 
We shall assume 
\begin{equation}
\label{G0.1}
\left\{ \begin{aligned}
   & |D_{\zeta} X| + |D_{\zeta} \chi| 
        \leq \varrho_0 |\zeta| + \varrho_1 \\
    &  |D_{\zeta} \psi| 
       \leq  \varrho_0 f (|\zeta|^2 \mathbf{1})/|\zeta|, \\
  &  D_\phi X, D_{\phi} \chi 
          \leq (\varrho_0 |\zeta|^2 + \varrho_1) \omega, \\       
   & D_{\phi} \psi 
        \geq - \varrho_0 f(|\zeta|^2 \mathbf{1}) - \varrho_1  
        \end{aligned} \right.     
\end{equation}
at $(z, \phi, \zeta,\bar {\zeta})$, where $\varrho_1=\varrho_1(z,\phi)$ and 
 $\varrho_0 = \varrho_0 (z, \phi, |\zeta|) \to 0^+ \;\; \mbox{as $|\zeta| \to \infty$}.$
It follows that 
\begin{equation}
\label{G0.2}
\left\{ \begin{aligned}
   & X (z, \phi, \zeta,\bar {\zeta}), \chi (z, \phi, \zeta,\bar {\zeta}) \leq (\varrho_0 |\zeta|^2 + \varrho_2 (z,\phi)) \omega, \\
   &  \psi (z, \phi, \zeta,\bar {\zeta}) 
          \leq \varrho_0 f (|\zeta|^2 \mathbf{1}) + \varrho_2 
      \end{aligned} \right.   
\end{equation}
for some function $\varrho_2$.
We shall also assume  
\begin{equation}
\label{G0.5}
    |\nabla_z X| + |\nabla_z \chi|  \leq  |\zeta| \left(\varrho_0 |\zeta|^2+\varrho_1\right) 
    \;\; 
    |\nabla_z \psi| \leq |\zeta| \left(\varrho_0 f (|\zeta|^2 \mathbf{1}) 
    + \varrho_1\right).    
\end{equation}


\begin{theorem}
There exists uniform constant $C$ depending on $\sup{|\phi_t|}$ such that 
\begin{equation}
\label{G0}
    |\nabla \phi| \leq C  \Big(1 + \sup_{M \times [0, T]} {\phi} - \phi\Big).
\end{equation}
\end{theorem}

\begin{proof}
We assume without loss of generality that 
\[ \sup\limits_{\partial \Gamma} f \leq 0 < \psi. \]
Let 
\[  \eta =1 + \sup\limits_{M\times [0,T)} \phi - \phi. \]
and suppose that the function $\eta^{-1} |\nabla \phi|^2$ attains its maximum 
at $(z_0,t_0) \in M \times (0,T]$ where $|\nabla \phi|\geq 1$. 
We choose local coordinates as before such that $g_{i\bar j}=\delta_{i\bar j}$, $T^k_{ij}=2\Gamma^k_{ij}$ and $\fg_{i\bj}$ are diagonal at $(z_0, t_0)$ where

\begin{equation}
\label{G1}
    \begin{aligned}
 \frac{\partial_i |\nabla \phi|^2}{|\nabla \phi|^2} + \frac{\partial_i \phi}{\eta} = 0, \;\; 
  \frac{\partial_\bi |\nabla \phi|^2}{|\nabla \phi|^2} + \frac{\partial_\bi \phi}{\eta} = 0,
    \end{aligned}
\end{equation}
\begin{equation}
\label{G2}
\begin{aligned}
\frac{\cL  |\nabla \phi|^2}{|\nabla \phi|^2}
+  \frac{\cL \phi}{\eta} \geq 0.
\end{aligned}
\end{equation}

Denote
\[ S_{i\bj} = \nabla_i \nabla_\bk \phi \nabla_k \nabla_\bj \phi 
                + \nabla_i \nabla_k \phi \nabla_\bj \nabla_\bk \phi.\]
By \eqref{G1} and Schwarz inequality,
\begin{equation}
\label{G7}
    \begin{aligned}
G^{i\bj} S_{i\bj} \geq \frac{1}{|\nabla \phi|^2} G^{i\bj} \partial_{i} |\nabla \phi|^2 \partial_\bj |\nabla \phi|^2
 = \frac{|\nabla \phi|^2}{\eta^2} G^{i\bj} \partial_i \phi \partial_\bj \phi   
       \end{aligned}
\end{equation}

We calculate 
\begin{equation}
\label{G4}
\begin{aligned}
   \cL  |\nabla \phi|^2 
          = \,& 2 \fRe \{\nabla_\bk \phi \cL \nabla_k \phi\} - G^{i\bj} S_{i\bj} \\
     \leq  \,&  2 \fRe\{B_k \nabla_\bk \phi 
       + B^{\alpha} \Gamma_{\alpha k}^l \nabla_l \phi \nabla_{\bk} \phi\} 
       +  2 B_{\phi} |\nabla \phi|^2 \\
     & - (1 - \gamma) G^{i\bj} S_{i\bj} + C |\nabla \phi|^2 \sum G^{i\bi} \\
    \leq  \,&  |\nabla \phi|^2 (C R - 2 c_1 G^{i\bj} \partial_i \phi \partial_\bj \phi)   
        \end{aligned}
\end{equation}
by assumptions~\eqref{G0.1} and \eqref{G0.5}, where $0< \gamma < 1$ and 
\[ R = (1 + \varrho_0 |\nabla \phi|^2) \sum G^{i\bi}
              + 1 + \varrho_0 f (|\nabla \phi|^2 \bf{1}). \]
Similarly, 
\begin{equation}
\label{G8}
\begin{aligned}
  \cL \phi = \,& \partial_t \phi - G^{i\bj} \fg_{i\bj} + G^{i\bj} \chi_{i\bj}
                        + 2 \fRe\{B^{\alpha} \partial_{\alpha} \phi\}        \\
               = \,& \partial_t \phi - F^{I\bJ} Z_{I\bJ} + F^{I\bJ} X_{I\bJ} + G^{i\bj} \chi_{i\bj}
                        + 2 \fRe\{B^{\alpha} \partial_{\alpha} \phi\}        \\ 
           \leq \,& \partial_t \phi - F^{I\bJ} Z_{I\bJ} + C R.                                   
\end{aligned}
\end{equation}
By concavity of $f$ as in the previous section, 
\begin{equation}
\label{G11}
    \begin{aligned}
 |\nabla\phi|^2 \sum G^{i\bi} 
        \geq \,& f (|\nabla\phi|^2 \mathbf{1}) - C \sum G^{i\bi} - C.
    \end{aligned}
\end{equation}

Finally, applying Lemma~\ref{lemma-P10} and Lemma~\ref{lemma 3}, 
by \eqref{G2} and \eqref{G4}-\eqref{G11} we derive
\begin{equation}
\label{G12}
f (|\nabla\phi|^2 \mathbf{1}) \Big(\sum G^{i\bi} + 1\Big) \leq C R
\end{equation}
which gives a bound $|\nabla\phi|^2 \leq C$.
\end{proof}

As a consequence we obtain a bound  for the oscillation of $\phi$.

\begin{corollary}
\label{cor1}
For $\phi$ as above,
$$\left|\left(1+\sup{\phi}-\phi(x,t)\right)^{\frac{1}{2}}-\left(1+\sup{\phi}-\phi(y,s)\right)^{\frac{1}{2}}\right|\leq C d$$
for any $(x,t)$, $(y,s)$ in $M\times [0,T)$, where $d$ is the diameter of $M$. 
In particular,
\begin{equation}
\label{G16}
    \sup{\phi}-\inf{\phi}\leq C\max\{d,d^2\}.
\end{equation}
\end{corollary}

From the estimates in this and the previous section we obtain the following existence result, proving the first part of Theorem~\ref{theorem-I1}. 

\begin{theorem}
\label{theorem-I11}
Let $\phi_0 \in C^{\infty} (M)$ be an admissible function with $c_0 [\phi_0] > 0$.
Assume that \eqref{P3}, \eqref{P4}, \eqref{S0.2}-\eqref{S0.1} 
\eqref{G0.1}, \eqref{G0.5} hold and $B_{\phi} \leq 0$ in \eqref{GG-E10}. 
There exists an admissible solution $\phi \in C^{\infty} (M \times \{t > 0\})$ 
of problem~\eqref{I1}-\eqref{gg-I20} for $\Omega = X [\phi]$.
\end{theorem}

\bigskip

\section{Harnack inequality and uniform convergence of solution}
\label{H}
\setcounter{equation}{0}

\medskip

In this section we prove the second part of Theorem~\ref{theorem-I1}. For this purpose, we extend Li-Yau Harnack inequality from ~\cite{LY86}, and \cite{Gill11} to equations with lower order terms. This follows the argument in \cite{George21}, but we give a self-contained proof here. 

We consider the following uniformly parabolic equation locally given by
\begin{equation}
    \label{H1}
    \begin{aligned}
    \frac{\partial u}{\partial t} =  G^{i\bar j} \partial_i \partial_{\bar j} u 
        +\chi^i \partial_i u  + \bar \chi^i \partial_{\bi} u + \chi^0 u
    \end{aligned}
\end{equation}
where 
$\chi^i, \chi^{0} \in C^{3,1} (M \times [0, T))$ and 
$G^{i\bj} \in C^{3,1} (M \times [0, T))$ with 
 \begin{equation}
    \label{H1.5}
  0 < \lambda g^{i\bar j} \xi_i \xi_{\bar j} \leq G^{i\bar j} \xi_i \xi_{\bar j} \leq \Lambda g^{i\bar j} \xi_i \xi_{\bar j}, \;\; \forall \, \xi \in T^{1,0} M
 \; \mbox{on $M \times [0, T)$.} 
 \end{equation} 
 
\begin{lemma}
\label{lemma2.1}
Let $u \in C^{3,2} (M\times [0,T))$ be a positive solution of \eqref{H1}, $\alpha > 1$.
There are constants $C_1$ and $C_2$ such that 
\begin{equation}
\label{H2}
\begin{aligned}
\frac{1}{u^2} G^{i\bar j} \partial_i u \partial_{\bar j} u - \frac{\alpha \partial_t u}{u} 
     \leq C_1 + \frac{C_2}{t}, \;\; 0 < t < T. 
\end{aligned}
\end{equation}
\end{lemma}

\begin{proof} 
Denote $h = \log u$ and consider the quantity 
\[ Q = G^{i\bar j} \partial_i h \partial_{\bar j} h 
        - \alpha \partial_t h. \]
For convenience we shall write 
$h_i = \partial_i h$, $h_{i\bj} = \partial_\bj \partial_i h$, 
$h_{i\bj k} = \nabla_k \nabla_{\bj} \nabla_i h$, etc.  

It follows form straightforward calculations that
\begin{equation}
\label{H6}
\begin{aligned}
    Q_t = 
    2 \fRe\{G^{i\bar j} h_{ti} h_{\bar j}\}  - \alpha h_{tt},
\end{aligned}
\end{equation}
\begin{equation}
\label{H7}
\begin{aligned}
    G^{i\bar j} Q_{i\bar j} 
    = \,& G^{i\bar j} G^{k\bar l} (h_{k\bar j} h_{\bar l i} 
    + h_{ki} h_{\bar l\bar j} + h_{ki \bar j} h_{\bar l} + h_{k} h_{\bar l i \bar j}) 
    - \alpha G^{i\bar j} h_{t i\bar j} + H 
    \end{aligned}
\end{equation}
where, for any constant $\epsilon > 0$,  
\begin{equation}
\label{H8}
\begin{aligned}
  H := \,& G^{i\bar j} (\partial_i \partial_{\bar j} G^{k\bar l} h_{k} h_{\bar l} 
              + \partial_i G^{k\bar l} h_{k\bar j} h_{\bar l} 
              + \partial_{i} G^{k\bar l} h_{k} h_{\bar l\bar j} 
              + \partial_{\bar j} G^{k\bar l} h_{\bar l} h_{ki} 
              + \partial_{\bar j} G^{k\bar l} h_{k} h_{\bar l i}) \\
    \geq \,& - \epsilon G^{i\bar j} G^{k\bar l} (h_{i\bar l} h_{k\bar j} + h_{ik} h_{\bar j \bar l})   
     - \Big(C + \frac{1}{2 \epsilon}\Big) G^{i\bar j} h_i h_{\bj}  
\end{aligned}
\end{equation} 
by Schwarz inequality.

From equation~\eqref{H1} we see that 
\begin{equation}
\label{H5}
\begin{aligned}
    Q = ( 1 - \alpha ) h_t -  G^{i\bar j} h_{i\bar j} - \chi^i h_i - \bar \chi^i h_{\bar i} - \chi^0.
\end{aligned}
\end{equation}
We derive using Schwarz inequality
\begin{equation}
\label{H14}
\begin{aligned}
 G^{i\bar j } h_{ti\bar j} 
    = \,& (1 - \alpha) h_{tt} - (Q + \chi^i h_i + \bar \chi^i h_{\bar i} +\chi^0)_t 
             -  \partial_t G^{i\bar j} h_{i\bar j} \\
 \leq \, &  (1 - \alpha) h_{tt} - Q_t - \chi^i h_{i t} - \bar \chi^{i} h_{\bi t} 
           +  \epsilon |\partial \bpartial h|^2 + C |\partial h|^2 + C
\end{aligned}
\end{equation}
where $|\partial h|^2 =g^{i\bar j} h_i h_{\bj}$, 
$|\partial \bpartial h|^2 = g^{i\bar j} g^{k\bar l} h_{i\bar l} h_{k\bar j}$, and 
\begin{equation}
\label{H9}
\begin{aligned}
   \fRe\{G^{i\bar j} G^{k\bar l} h_{ki\bar j} h_{\bar l}\} 
   \geq \,&  \fRe\{G^{i\bar j} G^{k\bar l} h_{i\bj k} h_{\bar l}\} 
                  - \epsilon G^{i\bar j} G^{k\bar l} h_{i\bar l} h_{k\bar j}  
                  - \frac{C}{\epsilon} G^{i\bar j} h_i h_{\bj} \\
   \geq \,& (1 - \alpha) \fRe\{G^{i\bj} h_{t i} h_{\bj}\} - \fRe\{G^{i \bj} Q_i h_{\bj}\} \\
             & - \epsilon G^{i\bar j} G^{k\bar l} (h_{i\bar l} h_{k\bar j} + h_{ik} h_{\bar j \bar l})
                - C G^{i\bar j} h_i h_{\bj} - C |\nabla \chi^0| \sum G^{i\bi}. 
\end{aligned}
\end{equation}
Note that by \eqref{H1.5},
\[ \lambda |\partial h|^2 \leq G^{i\bar j} h_i h_{\bj} \leq \Lambda |\partial h|^2, \;\;
 \lambda^2 |\partial \bpartial h|^2 \leq G^{i\bar j} G^{k\bar l} h_{i\bar l} h_{k\bar j} 
         \leq \Lambda^2 |\partial \bpartial h|^2.  \]
Similarly, for $|\partial \partial h|^2 = g^{i\bar j} g^{k\bar l} h_{ik} h_{\bj\bl}$,
 \[ \lambda^2 |\partial \partial h|^2 \leq G^{i\bar j} G^{k\bar l} h_{ik} h_{\bj\bl} 
         \leq \Lambda^2 |\partial \partial h|^2.  \]
 We shall use these facts without further references.

From~\eqref{H6}, ~\eqref{H14} and \eqref{H9} we obtain
\begin{equation}
\label{H7.5}
\begin{aligned}
 2 \fRe\{G^{i\bar j} G^{k\bar l} h_{ki \bar j} h_{\bar l}\} 
    - \alpha G^{i\bar j} h_{t i\bar j} 
\geq  \,& Q_t - 2 \fRe\{G^{i \bj} Q_i h_{\bj}\} 
               + \alpha (\chi^i h_{i t} +  \bar \chi^{i} h_{\bi t}) \\
            & -  C\epsilon |\nabla^2 h|^2 - C |\partial h|^2 - C 
    \end{aligned}
\end{equation}
where $|\nabla^2 h|^2 = |\partial \bpartial h|^2 + |\partial \partial h|^2$. 

Suppose now that $\tilde Q := t Q$ attains its maximum in $M \times [0,T']$ at a point 
$(x_0,t_0)$ in  $M \times (0,T']$ where $0 < T' < T$. 
We have at $(x_0,t_0)$,
\begin{equation}
\label{H3}
\begin{aligned}
     \alpha h_{kt}  = \,& \partial_k  G^{i\bar j} h_i h_{\bj} 
                                    + G^{i\bar j} (h_{ik} h_{\bj} + h_i h_{k\bj}), \;\; 1 \leq k \leq n
\end{aligned}
\end{equation}
which yields
\begin{equation}
\label{H3.5}
\begin{aligned}
  \alpha  |\chi^i h_{i t} + \bar \chi^{i} h_{\bi t}| 
    \leq \,& \epsilon |\nabla^2 h|^2 + C |\partial h|^2
      \end{aligned}
\end{equation}
and, by \eqref{H7}, ~\eqref{H8}, \eqref{H7.5} and \eqref{H3.5},
\begin{equation}
\label{H15}
\begin{aligned}
 0 \geq \,& G^{i\bar j} \tilde Q_{i\bar j} - \tilde Q_t 
   \geq  (\lambda^2 - C \epsilon) t |\nabla^2 h|^2 -  C t |\partial h|^2 - Q - C t.
\end{aligned}
\end{equation}
By \eqref{H5} and arithmetic-geometric inequality,
\begin{equation}
 \label{H16}
 \begin{aligned}
|\partial \bar \partial h|^2 
   \geq \,& \frac{1}{\Lambda^2} G^{i\bar j} G^{k\bar l} h_{i\bl} h_{k\bj} 
                 \geq \frac{1}{n \Lambda^2} (G^{i\bar j} h_{i\bar j})^2 \\
    \geq \,& c_0 (Q + (\alpha - 1) h_t)^2 
                 - C |\partial h|^2 - C  \\
        =  \,& c_0 (Q + (\alpha - 1) h_t)^2 - C_1 Q - C_2 \alpha h_t - C
 \end{aligned}
 \end{equation}
We derive from \eqref{H15} that at $(x_0, t_0)$, 
\begin{equation}
\label{H17}
\begin{aligned}
 (\tilde Q +  (\alpha - 1) h_t t_0)^2 + C \alpha h_t t_0^2 \leq C (1 + t_0^2).
\end{aligned}
\end{equation}
This immediately gives $\tilde Q \leq C (1 + t_0)$ provided that $h_t \geq 0$. 

Assume now $h_t < 0$ so $- h_t \leq Q + (\alpha - 1) h_t$. From \eqref{H17} we obtain 
\[ h_t^2 + C \alpha h_t \leq C (1 + t_0^{-2}). \]
Therefore, $ - \alpha h_t \leq C + C t_0^{-1}$. We derive  
$\tilde Q \leq C (1 + t_0)$ from \eqref{H17} again. 

Finally, for any $x \in M$ and $0 < t < T$, taking $T' = t$ we obtain
\begin{equation}
    \label{H22}
    \begin{aligned}
        T' Q(x, T') \leq \tilde Q (x_0, t_0) \leq C (1 + t_0) \leq C (1 + T'). 
    \end{aligned}
\end{equation}
This proves \eqref{H2}. 
\end{proof}

It is often convenient to use the following equivalent form of Lemma~\ref{lemma2.1}.

\begin{lemma}
\label{lemma2.2}
Let $u \in C^{3,2} (M\times [0,T))$ be a positive solution of \eqref{H1}, $\alpha > 1$.
There are constants $C_1$ and $C_2$ such that 
\begin{equation}
\label{H2a}
\begin{aligned}
|\partial \log u|^2 - \alpha \partial_t \log u
     \leq C_1 + \frac{C_2}{t}, \;\; 0 < t < T. 
\end{aligned}
\end{equation}
\end{lemma}

With the aid of Lemma~\ref{lemma2.2} we may derive the following Harnack inequality
as in Li-Yau~\cite{LY86}. For completeness we include an outline of the proof. 

\begin{theorem}
\label{theorem-H}
Let $u \in C^{3,2} (M \times [0,T))$ be a positive solution of \eqref{H1}, $\alpha > 1$.
Then 
\begin{equation}
    \sup_{z \in M} u (z, t) \leq C(t, \tau; \alpha) \inf_{z \in M} u(z, \tau)
\end{equation}
for $0 < t < \tau < T$, where
\begin{equation}
    C(t, \tau, \alpha) = \Big(\frac{\tau}{t}\Big)^{\beta} 
          \exp \Big(\frac{C_1}{\tau - t} + C_2 (\tau - t)\Big)
\end{equation}
and $C_1$, $C_2$, $\beta$ are positive constants depending on $n$, $\lambda$, $\Lambda$, and
\[ |G^{i\bj}|_{C^{3,1} (M \times [0, T))}, \;  |\chi^i|_{C^{1,1} (M \times [0,T))},  
    \; |\chi^0|_{C^{1,1} (M \times [0,T))} \]
as well as the geometric quantities of $M$.
\end{theorem}

\begin{proof}
For any two points $z, w \in M$, let $\gamma: [0, 1] \to M$ be a geodesic with  
$\gamma (0) = z$ and $\gamma(1)= w$. 
Define $\eta: [0, 1] \to M \times [t, \tau]$ 
by 
\[ \eta (s) = (\gamma (s), (1- s) \tau + s t), \; 0 \leq s \leq L. \]
For $h = \log u$ as before, by Schwarz inequality and Lemma \ref{lemma2.2} we have
\begin{equation}
    \begin{aligned}
     \log \frac{u (w, t)} {u(z, \tau)} 
        = \,& \int_0^1 \frac{d}{ds} h(\eta(s)) ds\\
        = \,& \int_0^1 (\langle \dot {\gamma}, \partial h \rangle_g - (\tau - t) h_t )ds \\
    \leq \,& \int_0^1 
                   \Big(\frac{\alpha |\dot{\gamma}|^2}{\tau - t} 
                   + \frac{\tau - t}{\alpha} (|\partial h |^2 - \alpha h_t)\Big) ds\\
    \leq \,& \frac{\alpha |\dot{\gamma}|^2}{\tau - t} 
               + \frac{\tau - t}{\alpha} \int_0^1 \Big(C + \frac{C}{(1-s) \tau + s t}\Big) ds \\
       = \,& \frac{\alpha |\dot{\gamma}|^2}{\tau - t}  + \frac{\tau - t}{C^{-1} \alpha}   
                +\frac{C}{\alpha} \log \frac{\tau}{t}. 
    \end{aligned}
\end{equation}
Note that $|\dot{\gamma}|$ is dominated by the diameter of $M$. The proof is complete.
\end{proof}

We can now use a standard iteration argument to prove the convergence of the 
normalized solution in Theorem~\ref{theorem-I1}. Assume that $X$, $\chi$ and $\psi$ are independent of $\phi$ (but may still depend on $\partial \phi, \bpartial \phi$). 
Let $\omega (t)$ be the oscillation of $\phi_t$ 
\[ \omega(t) := \sup\limits_{z \in M} \phi_t (z, t) - \inf\limits_{z \in M} \phi_t (z, t) \]
and for $k = 0, 1, 2, \ldots$,
\[ \begin{aligned}
     v_k (z, t) = \,& \sup\limits_{y \in M} \phi_t (y, k) - \phi_t (z, k + t), \\
    w_k (z, t) = \,& \phi_t (z, k + t) - \inf\limits_{y \in M} \phi_t (y, k). 
    \end{aligned} \]
It is clear that both $v_k$ and $w_k$ satisfy the following equation
\begin{equation}
\label{C2}
\begin{aligned}
    \frac{\partial \varphi}{\partial t} (z, t) 
      = G^{i\bj}(x, k+t) \partial_i \partial_{\bj} \varphi
       +\chi^i (x, k+t) \partial_i \varphi + \bar \chi^i (x, k+t) \partial_\bi \varphi 
\end{aligned}
\end{equation}

Suppose $\phi_t (z, k)$ is not constant so that $v_k$ and $w_k$ are positive somewhere in $M$ at time $t=0$.
Note that 
\[ \chi^0 = F^{I\bJ} X_{I\bJ, \phi} +  G^{i\bj} \chi_{i\bar j,\phi} - \psi_{\phi} \equiv 0. \]
 By the maximum principle $v_k$ and $w_k$ are both positive on $M$ for all $t > 0$. 
We may apply Theorem~\ref{theorem-H} to $v_k$ and $w_k$ 
with $t =\frac{1}{2}$ and $\tau = 1$ to derive
\[  \begin{aligned}
\sup\limits_{z \in M} \phi_t (z, k) - \inf\limits_{z \in M} \phi_t \Big(z, k + \frac{1}{2}\Big) 
    \leq \,& C \Big(\sup\limits_{z \in M} \phi_t (z, k) - \sup\limits_{z \in M} \phi_t (z, k+1)\Big), 
    \end{aligned} \]
\[    \begin{aligned}
\sup\limits_{z \in M} \phi_t \Big(z, k + \frac{1}{2}\Big) - \inf\limits_{z \in M} \phi_t (z, k) 
    \leq \,& C \Big(\inf\limits_{z \in M} \phi_t (z, k+1) - \inf\limits_{z \in M} \phi_t (z, k)\Big)
    \end{aligned} \]
where $C:=C(\frac{1}{2},1)$. 
Adding the two inequalities yields
\[  \omega (k) + \omega \Big(k+\frac{1}{2}\Big) \leq C (\omega (k) - \omega (k+1)). \]
It follows that
\begin{equation}
\label{C4}
\omega(k+1) \leq \delta \omega(k), \;\; \delta < 1. 
\end{equation}
This still holds by the by maximum principle if $\phi_t (z, k)$ is constant. 
Using the standard iteration argument we derive 
$\omega (t) \leq C e^{-\beta t}$ for $\beta = - \log {\delta}$. 

Note that $\tilde{\phi}_t (y, t) = 0$ for some $y \in M$
since $\int_M \tilde{\phi}_t \omega^n = 0$. Therefore,
\begin{equation}
\label{C5}
    \begin{aligned}
    |\tilde{\phi}_t (z, t)| =  |\tilde{\phi}_t (z, t) - \tilde{\phi}_t (y, t)|
    \leq \omega (t) \leq Ce^{-\beta t}
    \end{aligned}
\end{equation}
which implies
\[ \frac{\partial }{\partial t} (\tilde{\phi} (z, t) - C \beta^{-1} e^{-\beta t}) \leq 0. \]
Consequently, the limit 
\[ \varphi (z) := \lim_{t \rightarrow \infty} \tilde{\phi} (z, t) 
     = \lim_{t \rightarrow \infty} (\tilde{\phi} (z, t) - C \beta^{-1} e^{-\beta t}) \] 
exists for all $z \in M$ as $\tilde{\phi} (z, t) - C \beta^{-1} e^{-\beta t}$ is uniformly bounded.
Finally, from the {\em a priori} estimates we see that $\varphi \in C^{\infty} (M)$ and is an admissible solution of the elliptic equation~\eqref{I4}. 
The proof of Theorem~\ref{theorem-I1} is complete.

\bigskip

\bibliographystyle{plain}

\end{document}